\UseRawInputEncoding
\documentclass[11pt,twoside,a4paper]{amsart}
\usepackage{amssymb}

\usepackage{amsmath}
\usepackage{amsthm,amssymb,latexsym}
\usepackage[all,cmtip]{xy}
\usepackage{url,ae}
\usepackage{t1enc}
\usepackage{fancyheadings}
\usepackage{mathrsfs}
\usepackage{graphicx}
\usepackage{eufrak} 
\usepackage{geometry}

\date{\today}

\def\1{{\bf 1}}

\def\GZ{{\mathcal{GZ}}}
\def\Zy{{\mathcal Z}}

\def\deg{\text{deg}\,}

\def\w{\wedge}

\def\C{{\mathbb C}}
\def\w{{\wedge}}
\def\P{{\mathbb P}}

\def\bl{{\mathcal B}}
\def\A{{\mathcal A}}
\def\B{{\mathcal B}}

\def\codim{{\rm codim\,}}

\def\Z{{\mathbb Z}}

\def\J{{\mathcal J}}
\def\nbh{neighborhood }

\def\be{\begin{equation}}
\def\ee{\end{equation}}

\def\Ok{\mathcal O}
\def\mult{{\rm mult}}

\def\Pk{{\mathbb P}}

\def\di{{\diamond}}
\def\bl{{\bullet_L}}

\newtheorem{thm}{Theorem}[section]
\newtheorem{lma}[thm]{Lemma}

\newtheorem{prop}[thm]{Proposition}

\theoremstyle{definition}

\newtheorem{df}[thm]{Definition}

\theoremstyle{remark}

\newtheorem{preremark}[thm]{Remark}
\newtheorem{preex}[thm]{Example}

\newenvironment{ex}{\begin{preex}}{\qed\end{preex}}

\numberwithin{equation}{section}

\usepackage{xcolor}

\title[On non-proper intersections and local intersection numbers]{On non-proper intersections and 
local intersection numbers}
\begin{document}

\date{\today}

\author[Andersson \& Samuelsson Kalm \&
Wulcan]
{Mats Andersson \&
H\aa kan Samuelsson Kalm 
\& Elizabeth Wulcan} 

\address{Department of Mathematical Sciences\\Chalmers University of Technology and University of
Gothenburg\\SE-412 96 G\"OTEBORG\\SWEDEN}

\email{matsa@chalmers.se, hasam@chalmers.se, wulcan@chalmers.se}


%\thanks{The first, third and fourth  author  were
 % partially supported by the Swedish
 % Research Council}
%%the second author was partially supported by
%%  a post doctoral fellowship from the Swedish
%%  Research Council}

\begin{abstract}
Given pure-dimensional (generalized) cycles $\mu_1$ and $\mu_2$ on a complex manifold $Y$ 
we introduce a product $\mu_1\diamond_{Y} \mu_2$ that is a generalized cycle whose multiplicities
at each point are the local intersection numbers at the point.  
If $Y$ is projective, then given a very ample line bundle $L\to Y$ we define a product $\mu_1\bl \mu_2$ 
whose multiplicities at each point also coincide with the local intersection numbers.
In addition, provided that $\mu_1$ and $\mu_2$ are effective, this product satisfies a B\'ezout inequality.
If $i\colon Y\to \Pk^N$ is an embedding such that $i^*\Ok(1)=L$, then $\mu_1\bl \mu_2$
can be expressed as a mean value of  St\"uckrad-Vogel  cycles on $\Pk^N$.   
There are quite explicit relations between $\di_Y$ and  $\bl$.
 \end{abstract}

%%%%%%%%%%

\maketitle

\section{Introduction}

Let $Y$ be a complex manifold of dimension $n$.  A cycle on $Y$ is a 
locally finite linear combination, over $\Z$, of subvarieties of $Y$.   
Assume that $\mu_1$ and $\mu_2$ are equidimensional cycles on $Y$. 
If they intersect properly, i.e., the expected dimension
$\rho=\dim \mu_1+\dim \mu_2-n$ is equal to the 
dimension of their set-theoretical intersection $V=|\mu_1|\cap|\mu_2|$,
 %%
%then If $\mu_j$ intersect properly, that is, if $\codim V=\kappa$, 
then there is a well-defined intersection cycle
\begin{equation*}%\label{basal}
\mu_1\cdot_Y  \mu_2=\sum m_j V_j,
\end{equation*}
where $V_j$ are the irreducible components of $V$ and $m_j$ are integers.
 %%%
If  $\mu_1$ and $\mu_2$ do not intersect properly, i.e., $\dim V > \rho$,  following \cite{Fult},
the product $\mu_1\cdot_Y  \mu_2$ is represented by a cycle of
dimension $\rho$ on $V$ that is determined up to rational equivalence, i.e.,  
a Chow class on $V$. 
In case $Y=\P^n$ there is a construction of a product $\mu_1\cdot_{SV}\mu_2$
due to St\"uckrad and Vogel  that is represented by a cycle on $V$ with components of various degrees.
This cycle, which we call a SV-cycle, is obtained by a quite explicit procedure, which however 
involves various choices. 
By van~Gastel's formula, \cite{gast},  one can obtain the Chow class $\mu_1\cdot_Y  \mu_2$
from a generic representative of $\mu_1\cdot_{SV}\mu_2$.  
\smallskip

In the '90s Tworzewski, \cite{Twor}, introduced {\it local intersection numbers}
$\epsilon_\ell(\mu_1,\mu_2,x)$ at each point $x$,  $0\le \ell\le \dim V$, which reflect the
complexity of the intersection at dimension $\ell$.  In particular, if the intersection is proper, then
$\epsilon_\ell(\mu_1,\mu_2,x)$ is the multiplicity of $\mu_1\cdot_Y\mu_2$ at $x$ when $\ell=\dim V$
and $0$ otherwise.
In this case thus all these numbers are represented by the global   cycle
$\mu_1\cdot_Y\mu_2$.
In general however 
there is no single cycle whose multiplicities of 
its components of various dimensions are precisely the local intersection numbers for all points
of $Y$, that is,  which represents all these local intersection numbers.

Since both the local and global intersections are defined within algebraic geometry, it is natural from
such a point of view to look for a way to unify these theories. However since this cannot be done by cycles
it is  natural to look for slightly more general geometric objects that may have the desired local multiplicities and 
at the same time in a reasonable sense represent the global intersection products.  
 
To this end, in \cite{aeswy1}, together with Eriksson and Yger, we introduced, for any reduced analytic space $X$, the group
$\B_k(X)$ of {\it generalized cycles} of dimension $k$, modulo a certain equivalence relation, that contains
the group $\Zy_k(X)$ of cycles of dimension $k$ as a subgroup. The {\it generalized cycles classes} in $\B_k(X)$ share many properties
with (usual) cycles. 
For instance, each $\mu\in\B_k(X)$ has a well-defined (integer) multiplicity 
$\mult_x\mu$ at each point $x\in X$ and a Zariski support $|\mu|$.  Each generalized cycle class 
is a unique sum of irreducible generalized cycle classes. 
Moreover, the generalized cycle classes $\mu$ in $\B_k(X)$ that have Zariski support 
$|\mu|$ on a subvariety $Z\subset  X$ are naturally identified with $\B_k(Z)$; 
see Section~\ref{kung} below for precise definitions and statements.
One can think of generalized cycle classes as mean values of cycles.  
We let $\B(X)=\oplus_0^m\B_k(X)$ if $m=\dim X$. 

\smallskip
For $\mu_1,\mu_2\in\B(\Pk^n)$, we defined with Eriksson and Yger, \cite{aeswy2},
an element  $\mu_1\bullet \mu_2\in \B(\Pk^n)$
that is equal to $\mu_1\cdot_{\Pk^n}\mu_2$ if the intersection is proper,
and whose  multiplicities at each point coincide with the local intersection numbers. If 
$\mu_j$ have pure dimensions and the expected dimension
\begin{equation}\label{rho}
\rho:=\dim\mu_1+\dim\mu_2-n
\end{equation}
is non-negative, then we have the B\'ezout equality
\begin{equation}\label{bezout}
\deg (\mu_1\bullet\mu_2)=\deg\mu_1\cdot\deg\mu_2.
\end{equation}
Roughly speaking,
 $\mu_1\bullet \mu_2$ is defined as a mean value of SV-cycles $\mu_1\cdot_{SV}\mu_2$ in case
 $\mu_j$ are cycles.

\smallskip
In this paper we  introduce two global intersection products that both respect all the 
local intersection numbers. The first one is defined on an arbitrary complex manifold.
The second one generalizes the $\bullet$-product and satisfies a B\'ezout inequality, but it is only defined on
projective manifolds.

%In our first main result in this paper, $Y$ is an arbitrary complex manifold.
 Here is our first main theorem.

\begin{thm}\label{main1}
Let $Y$ be a complex 
%\footnote{It is enough that the diagonal in $Y\times Y$ is defined by a holomorphic section of
%a vector bundle over $Y\times Y$.} 
% 
manifold. There is a  $\Z$-bilinear commutative pairing
$$
\B(Y)\times \B(Y)\to\B(Y), \quad (\mu_1,\mu_2)\mapsto \mu_1\di_Y \mu_2,
$$
with the following properties:

\smallskip
\noindent
(i)  $\mu_1\di_Y\mu_2$ has Zariski support on 
$$
V:=|\mu_1|\cap |\mu_2|.
$$

\smallskip
\noindent
(ii)  For each $x\in Y$ we have 
\begin{equation}\label{stare}
\epsilon_\ell(\mu_1,\mu_2,x) =\mult_x(\mu_1\di_Y\mu_2)_\ell,  \quad 0\le\ell\le \dim V,
\end{equation}
where $(\ \  )_\ell$ denotes the component of dimension $\ell$.

\smallskip
\noindent
(iii)  If $\mu_1,\mu_2$ are cycles that intersect properly, then 
$\mu_1\di_Y\mu_2=\mu_1\cdot_Y\mu_2$.

\smallskip
\noindent
(iv)  The natural image in the cohomology group $\widehat H^{*,*}(V)$ of the component 
$(\mu_1\di_Y\mu_2)_\rho$ of the expected dimension $\rho$, cf.~\eqref{rho},
coincides with the image of the Chow class $\mu_1\cdot_Y\mu_2$.
\end{thm}

For the definitions of the cohomology groups $\widehat H^{*,*}(V)$, see Section~\ref{prel}. 
Part (ii) means that $\di_Y$ solves our representation problem. However, (iv) suggests that already the component
of dimension $\rho$ is as `big' as the Chow class in a cohomological sense. In particular, if $Y=\Pk^n$
this implies that 
$$
\deg(\mu_1\di_{\Pk^n}\mu_2)_\rho=\deg (\mu_1\cdot_{\Pk^n}\mu_2),
$$
which by B\'ezout's equality and \eqref{bezout} is equal to $\deg(\mu_1\bullet\mu_2)$  if $\rho\ge 0$.
In general, the degree of the full generalized cycle class is much larger. For effective
generalized cycle classes, in particular for cycles, we have the estimate, see Section~\ref{kongas},
\begin{equation}\label{buske}
\deg (\mu_1\di_{\Pk^n}\mu_2)\le 2^{\dim V-\rho}\deg (\mu_1\bullet \mu_2).
\end{equation}
The constant, which is the best possible,  blows up when the intersection is far from being proper,
i.e.,  when $\dim V-\rho$ is large.
%However, if $\rho\ge 0$, then $ \deg (\mu_1\bullet \mu_2)$ is precisely equal to $\deg\mu_1\cdot\deg \mu_2$.

\smallskip 
It is thus natural to look for an extension of the $\bullet$-product, in order to get a representation
of the local intersection numbers that is not `too big'. To do this we restrict to
a projective manifold $Y$.   Let $L\to Y$ be a very ample line bundle. 
By definition then there is an embedding $i\colon Y\to \Pk^M$ 
for some $M$ such that $L=i^*\Ok(1)_{\Pk^M}$.  
Given $\mu_1,\mu_2 \in \mathcal{B}(Y)$ we 
can define a product $\mu_1\bl\mu_2$ such that 
\begin{equation}\label{lbull} 
i_*( \mu_1\bl\mu_2)=i_*\mu_1\bullet i_*\mu_2.
\end{equation}
For $\mu\in\B_k(Y)$ we define
\begin{equation}\label{kulting}
\deg_L\mu=\int_Y \mu\w c_1(L)^k
\end{equation}
%where $\omega_L=c_1(L)$. 
%
and extend to general $\mu\in\B(Y)$ by linearity.
Here is our second main result.

\begin{thm} \label{main2} 
Assume that $Y$ is a projective manifold and let $L\to Y$ be a very ample line bundle.
%and $\bl_Y$ defined by \eqref{lbull}.  

\smallskip
\noindent  
(i)
The pairing  
$
\B(Y)\times \B(Y)\to \B(Y), \  (\mu_1,\mu_2)\mapsto \mu_1\bl \mu_2,
$
defined by \eqref{lbull} is commutative and $\Z$-bilinear. It  depends on 
the choice of $L$ but not  on the embedding $i$.

\smallskip
\noindent  
(ii)   The product $\mu_1\bl\mu_2$ has Zariski support on $V=|\mu_1|\cap|\mu_2|$.

\smallskip
\noindent  
(iii)
For each $x\in Y$ we have
\begin{equation}\label{stenstod}
\mult_x(\mu_1\bl\mu_2)_\ell=\epsilon_\ell(\mu_1, \mu_2, x), \quad  \ell=0,1,\ldots, \dim V,
\end{equation}
where $(\ \  )_\ell$ denotes the component of dimension $\ell$.
 %
%\smallskip
%\noindent  
%(v)

\smallskip
\noindent  
(iv)
If the $\mu_1$ and $\mu_2$ are effective, then $\mu_1\bl \mu_2$  is effective and 
\begin{equation}\label{olikhet}
\deg_L(  \mu_1\bl \mu_2)\le \deg_L\mu_1\cdot\deg_L\mu_2.
\end{equation}
%
%$d:=\dim\mu_1+\cdots+\dim\mu_r-(r-1)N\ge 0$.  
%
%In case $\mu_j$ are
%ordinary cycles that intersect properly,  \eqref{poker} is the classical intersection product.
%
\smallskip
\noindent  
(v)
If $\mu_1, \mu_2$ are cycles that intersect
  properly, then   \begin{equation}\label{studsmatta}
\mu_1\bl\mu_2 = \mu_1\cdot_{Y}\mu_2 + \cdots,
\end{equation}
where $\cdots$ are terms with lower dimension and vanishing multiplicities. 
\end{thm}

In view of (iii) thus $\mu_1\bl\mu_2$ has the `right'  multiplicities at each point, whereas
(iv) says that we have control of the total  mass of $\mu_1\bl\mu_2$.
In case $Y=\P^n$
and $L=\Ok(1)$, then $\bullet_L$ coincides with $\bullet$. In this case 
the dots in \eqref{studsmatta} vanish.  However, in general they do not, see Example~\ref{kokong} in Section~\ref{examples}.  

 \smallskip
 
% \begin{remark}
% Theorem~1.1 

 %\smallskip
 
The plan of this paper is as follows.  In Section~\ref{prel} we recall necessary material from
\cite{aeswy1,aeswy2}. The $\di_Y$-product is defined in 
Section~\ref{di} and Theorem~\ref{main1} is proved. 
The relation to the $\bullet$-product on $\Pk^n$
is discussed in Section~\ref{kongas}.
%is We begin with rec
%%
In Section~\ref{bl} we prove Theorem~\ref{main2} and provide formulas that relate $\di_Y$ and $\bl$.
 In Section~\ref{further} we provide some further properties of these products,
 and in the final section, Section~\ref{examples},  we give various explicit examples.

\smallskip
\noindent {\bf Acknowledgment} \  We would like to thank Bo Berndtsson, Martin Raum and Jan Stevens 
for valuable discussions on questions in this paper.  We would also like to thank the referee for 
important comments and suggestions.

\section{Preliminaries}\label{prel}

Throughout this section $X$ is a reduced analytic space of dimension $n$. We let 
$\Zy_k(X)$ denote the $\Z$-module of $k$-cycles on $X$. 
Given $\mu\in\Zy_k(X)$ there is the  associated closed current $[\mu]$, the Lelong current,
of bidegree $(n-k,n-k)$. We will often identify $\mu$ and its Lelong current. 
%Recall that
%the the multiplicity $\mult_x \mu$ is equal to the Lelong number of $[\mu]$ at $x$.
If nothing else is stated the definitions and results in this section are
from  \cite[Sections~3 and 4]{aeswy1}.

\subsection{Generalized cycles}\label{kung}
The group $\GZ_k(X)$ of generalized cycles of dimension $k$ was introduced in \cite{aeswy1}. It is the
$\Z$-module generated by (closed) $(n-k,n-k)$-currents of the form
$\tau_*\alpha$, where
$\tau\colon W\to X$ is a proper mapping and 
\begin{equation}\label{alpha}
\alpha=\hat c_1(L_1)\w\ldots\w \hat c_1(L_r),
\end{equation}
where $L_j\to W$ are Hermitian\footnote{Throughout this paper all 
Hermitian metrics are smooth.} line bundles, and $\hat c_1(L_j)$ are the associated
first Chern forms.  We let $\GZ(X)=\oplus_0^n\GZ_k(X)$.  Here $W$ can be any 
complex variety but by virtue of Hironaka's theorem we may assume that $W$
is a connected manifold.  It is clear that generalized cycles are closed currents of order $0$.
Moreover, their Lelong numbers (multiplicities, see below) are integers. This means that 
$\GZ(X)$ is a quite restricted class of closed currents. 
We are basically interested in a certain quotient space $\B(X)=\oplus_0^n \B_k(X)$, where
$\B_k(X)$ are quotient spaces of $\GZ_k(X)$.  For precise definitions  and proofs
of the properties listed below, see \cite[Sections~3 and 4]{aeswy1}.  

%\begin{itemize} 
\smallskip
\noindent
(i) We have a natural inclusion $\Zy_k(X)\to \B_k(X)$ for each $k$ and hence an inclusion
$\Zy(X)=\oplus_0^n\Zy_k(X)\to \B(X)=\oplus_0^n\B_k(X)$.

\smallskip
\noindent
(ii) Each $\mu\in \B_k(X)$ has a well-defined Zariski support $|\mu|$; it is the smallest Zariski closed set
such that  $\mu$ has a representative in $\GZ_k(X)$ that vanishes in its complement. 

\smallskip
\noindent
(iii) Given $\mu\in\B(X)$ also its restriction $\1_V\mu$ to  the subvariety $V\subset X$ is an element in $\B(X)$.

\smallskip
\noindent
(iv) If $f\colon X\to X'$ is a proper mapping, then the push-forward $f_*$ induces a mapping
$f_*\colon \B_k(X)\to \B_k(X')$ that coincides with the usual push-forward on cycles. 

\smallskip
\noindent
(v) If $i\colon X\to X'$ is an embedding, then $i_*\colon \B_k(X)\to \B_k(X')$ is injective, and the image
is precisely the elements in $\B_k(X')$ with Zariski support on $i(X)$.  

\smallskip
\noindent
(vi) If $E\to X$ is a vector bundle, then we have natural mappings $c_k(E)\colon \B_*(X)\to \B_{*-k}(X)$.
The image of $\mu$ is represented by $\hat c_k(E)\w\hat\mu$, where $\hat\mu\in\GZ(X)$ represents
$\mu$ and $\hat c_k(E)$ is the Chern form associated with a (smooth) Hermitian metric on $E$. 

\smallskip
\noindent
(vii) If $f\colon X\to X'$ is a proper mapping and $E'\to X'$ is a vector bundle, then 
$f^*c_k(E')=c_k(f^*E')$,  and if $\mu\in \B_*(X)$, then
\begin{equation}\label{bus}
f_*(f^*c_k(E')\w\mu)=c_k(E')\w f_*\mu.
\end{equation}

\smallskip
\noindent
(viii) If $\mu\in\B(X)$ and $\mu'\in\B(X')$, where $X'$ is another reduced analytic space, then there is a well-defined
element  $\mu\times\mu'\in\B(X\times X')$, see \cite[Lemma~2.1]{aeswy2}. 

%\end{itemize}
\smallskip
In the recent paper \cite{Yger}  A.~Yger introduces the related notion of algebraic generalized cycle as a generalization
of (complex) algebraic cycle.

\subsection{Irreducibility}
A generalized cycle class $\mu\in \B(X)$ is irreducible if its Zariski support $|\mu|$ is an irreducible subvariety and 
$\mu$  has a representative
$\hat\mu$ with Zariski support $|\mu|$ such that $\1_W\hat\mu=0$ for each subvariety $W\subset X$ 
that does not contain $|\mu|$.
This condition on $\hat\mu$ is equivalent to that $\hat\mu$ is a (finite) sum of elements
of the form $\tau_*\alpha$, where $\alpha$ is a form as in \eqref{alpha} on $W$,  and
$\tau\colon W\to |\mu|$ is surjective. Notice that
these various terms can have different dimensions.

\smallskip
\noindent
Each element
in $\GZ(X)$, and in $\B(X)$, has a unique decomposition in irreducible components
with different Zariski supports. Each irreducible element  has in turn a unique decomposition
in components of various dimensions. 
%We let $\B_k(X)$ denote the elements in $\B_k(X)$ of pure dimension $k$.
% of various dimensions.
 %%
 
\smallskip
\noindent
There is a unique
decomposition
\begin{equation}\label{deco}
\mu=\mu_{fix}+\mu_{mov},
\end{equation}
where $\mu_{fix}$ is an ordinary cycle, whose irreducible components are called
the {\it fixed} components of $\mu$,
and $\mu_{mov}$, whose irreducible components are the {\it moving} components. Each  
moving component has strictly lower dimension than its Zariski support.
 
%which Zariski-supports have strictly larger dimension than the componentsger than their dimensions.  

%\smallskip

\subsection{Multiplicities} \label{multi}
If $\mu$ is a cycle,
then the multiplicity $\mult_x\mu$ at $x\in X$ is precisely the Lelong number at $x$ of
the associated Lelong current. If $X$ is not smooth, then 
$\mult_x\mu=\mult_{i(x)} i_*\mu$ if $i\colon X\to X'$ is an embedding and $X'$ is smooth.
There is a suitable definition of Lelong number that extends
to all generalized cycles and it turns out to depend only of their classes in $\B(X)$,
see \cite[Section~6]{aeswy1}. 
In this way we have for each  $\mu\in \B_k(X)$ 
well-defined multiplicities $\mult_x\mu$ at all points $x\in X$, and these numbers are integers.
They are local in the following sense:
 If $U\subset X$ is an open subset, then we have natural restriction mappings
$r_U\colon \B_k(X)\to \B_k(U)$, and $\mult_x\mu=\mult_x r_U\mu$.
 
If $i\colon X\to X'$, where $X'$ is smooth and $\mu\in\B(X)$, then
\begin{equation}\label{tosca50}
\mult_x\mu=\mult_{i(x)} i_*\mu.
\end{equation}
%
%\beginlma}
Assume that $\mu=\gamma\w \mu'$,where $\mu,\mu'\in\B(U)$ and $\gamma$ is smooth and has positive degree. 
Then $\mult_x\mu =0$.

\subsection{Effective generalized cycle classes}
In \cite[Section~2.4]{aeswy2} was introduced the notion of effective generalized cycle class $\mu\in \B(X)$
generalizing the notion of effective cycle.  It means precisely that $\mu$ has a representative $\hat\mu\in\GZ(X)$
that is a positive current. 
Effective generalized cycle classes have non-negative multiplicities at each point.

\subsection{The cohomology groups $\widehat H^{*,*}(X)$}\label{coho}
We define $\widehat H^{*,*}(X)$ as the vector space of closed $(*,*)$-currents of order $0$ modulo
the subspace generated by all 
$d\tau$ for  currents $\tau$ of order $0$ such that also $d\tau$ has order $0$,
cf.~\cite[Section~10]{aeswy1}.

\smallskip
\noindent
If $f\colon X\to X'$ is proper and $n'=\dim X'$, then we have natural mappings
$f_*\colon \widehat H^{n-*,n-*}(X)\to \widehat H^{n'-*,n'-*}(X')$.

\smallskip
\noindent
If $X$ is smooth, then $\widehat H^{*,*}(X)$ is naturally isomorphic to the usual cohomology
groups   $H^{*,*}(X,\C)$.

\smallskip
\noindent
For each $k$ there is a natural mapping $r_k\colon \Zy_k(X)\to \widehat H^{n-k,n-k}(X)$ 
that takes $\mu\in\Zy_k(X)$ to its Lelong current $[\mu]$.  This mapping extends to a mapping
$\B_k(X)\to \widehat H^{n-k,n-k}(X)$.  

\smallskip
\noindent
Each $\mu\in \Zy_k(X)$ defines an element in the Chow group $\A_k(X)$ and the mapping $r_k$
induces a mapping $\A_k(X)\to \widehat H^{n-k,n-k}(X)$.

\subsection{The $\B$-Segre class}\label{sklass}
Assume that $\J$ is a coherent ideal sheaf on $X$ with zero set $Z$. Also assume 
that  $\J$ is generated by a holomorphic section $\sigma$ of a Hermitian 
vector bundle $E\to X$.  That is, $\J$ is locally generated by the tuple of holomorphic functions
obtained when $\sigma$ is expressed in a local frame of $E$.  
Such a section $\sigma$ exists if $X$ is projective. 
%et us also assume
%thatThis is the case when $X$ is projective.   
For any $\mu\in\B(X)$, following \cite[Section~5]{aeswy1}, let
\begin{equation}\label{taket0}
M^\sigma_k\w \hat\mu=\1_Z(dd^c\log|\sigma|^2)^k\w\hat \mu:=
\1_Z \lim_{\epsilon\to 0} (dd^c\log(|\sigma|^2+\epsilon))^k\w\hat \mu, \quad k=0,1,\ldots,
\end{equation}
where $\hat\mu$ is a representative of the class $\mu$. The existence of the limit is highly non-trivial
and relies on a resolution of singularities. Then $M_k^\sigma\w\hat\mu$ defines a class 
$S_k(\J,\mu)$ in $\B_\ell(X)$, $\ell=\dim\mu-k$,
that only depends on $\mu$ and $\J$.  Clearly
$\mu\mapsto S_k(\J,\mu)$ is $\Z$-linear.  Let 
$S(\J,\mu)=S_0(\J,\mu)+\cdots +S_{\dim\mu} (\J,\mu)$.

\smallskip
\noindent
Let $M^\sigma\w\hat\mu= M_0^\sigma\w\hat\mu+\cdots+M_{\dim\hat\mu}^\sigma\w\hat\mu$.
If  $f\colon X\to X'$ is proper, $\J$ is an ideal sheaf on $X'$, and $\mu\in\B(X)$, then 
$f_* (M^{f^*\sigma}\w\hat\mu) = M^\sigma \w f_* \hat\mu$
and hence\footnote{In this note $f^*\J$ denotes the sheaf over $X$ generated by
 pullbacks of sections of $\J$.}
\begin{equation}\label{trottoar}
f_* S(f^*\J, \mu)=S(\J, f_*\mu).
\end{equation}
%For the definition of $S(\J,\mu)$ when we have no access to a section $\sigma$ as above, see \cite[Section 2.5]{aeswy2}.

\smallskip
If $\J$ is locally a complete intersection, that is, defines a regular embedding, then one can define the Segre classes
$S(\J,\mu)$ without a section $\sigma$ as above. We show this in Section~\ref{di} below in the case that
$\J$ is the ideal sheaf of a submanifold $V\subset X$. 
%See the proof of Theorem~\ref{main1} above, 
%where this is done for the ideal sheaf $\J$ that defines the diagonal in $Y\times Y$, where $Y$
%is a manifold.  

\subsection{Segre numbers}
Given a coherent ideal sheaf $\J\to X$ with zero set $Z$, and $\mu\in\B(X)$, there are, at each point $x$,
non-negative integers $e_k(\J,X,x)$ for $k=0,1,\ldots,\dim Z$, called the Segre numbers. They were 
introduced
independently by Tworzewski, \cite{Twor},  and Gaffney-Gassler, \cite{GG}, 
as the multiplicity of the component of codimension $k$ of a generic local SV-cycle in $\J_x$.
A purely algebraic definition was introduced in \cite{AM} and the
equivalence to the geometric definition was proved in \cite{AR}.  
If $Z$ is a point, then the Segre number is precisely the Hilbert-Samuel multiplicity.
In \cite{aswy} was introduced an analytic definition. 

Given $\mu\in\B(X)$ we have the integers
\begin{equation}\label{taket}
e_k(\J,\mu,x):=\mult_x S_k(\J, \mu),
\end{equation}
 %cf.~\cite[Definition~2.4]{aeswy1},
that are called the {\it Segre numbers of $\J$ on $\mu$} in \cite[Section~2.6]{aeswy1}.
If $\sigma$ is a section of a Hermitian vector bundle that defines $\J$ and $\hat\mu$ is a representative
of $\mu$, then 
\begin{equation}\label{taket22}
e_k(\J,\mu,x)=\mult_x M_k^\sigma\w\hat\mu. 
\end{equation}
Locally we can choose $\sigma$ and the (smooth) Hermitian metric so that $\log|\sigma|^2$ is 
plurisubharmonic.  If follows from \eqref{taket22} and \eqref{taket0}, and the Skoda-El~Mir theorem, that the Segre numbers
$e_k(\J,\mu,x)$ are non-negative if $\mu$ is effective.  We have that $e_k(\J,X,x)=e_k(\J,\1_X,x)$,
see \cite{aeswy1}.

\subsection{Local intersection numbers}\label{tyken}
Let $X$ be smooth, assume that $\mu_1,\mu_2\in \B(X)$ have pure dimensions,
and let $d=\dim\mu_1+\dim\mu_2$.
Furthermore, let  $\J_\Delta$ be the sheaf that defines the diagonal $\Delta$ in $X\times X$ and let
$j\colon X\to X\times X$ be the natural parametrization.   We define the 
{\it local intersection numbers}
\begin{equation}\label{taket12}
\epsilon_\ell(\mu_1,\mu_2,x)= e_{d-\ell}\big(\J_\Delta,\mu_1\times\mu_2, j(x)\big), \quad \ell=0,1,\ldots,
\end{equation}
saying that $\epsilon_\ell(\mu_1,\mu_2,x)$ is the local intersection number at dimension $\ell$.
These numbers are biholomorphic invariants, and if we have an embedding $i\colon X\to X'$
in a larger manifold $X'$, then it follows from \eqref{trottoar}, \eqref{taket} and \eqref{taket12} that
\begin{equation}\label{sukta}
\epsilon_\ell(\mu_1,\mu_2,x)=\epsilon_\ell(i_*\mu_1, i_*\mu,i(x))
\end{equation}
for each $x\in X$.

\section{The  $\di_Y$-product}\label{di}
Let $Y$ be a complex manifold, let
%$\Delta$ be the diagonal in $Y\times Y$, and let
$j\colon Y\to Y\times Y$ be the natural parametrization of the diagonal $\Delta$ in $Y\times Y$, and let $\J_\Delta$ be the corresponding ideal sheaf.
To define the $\di_Y$-product $\mu_1\di_Y\mu_2$ for $\mu_1,\mu_2\in\B(Y)$ we need a definition of $S(\J_\Delta, \mu_1\times \mu_2)$
when $\J_\Delta$ is not necessarily generated by a global holomorphic section of a Hermitian vector bundle $E\to Y\times Y$.

\smallskip

Let $X$ be a complex manifold, $i\colon V\to X$ a submanifold, and $\J_V$ the corresponding ideal sheaf. 
Recall that if there is a holomorphic section $\sigma$ of a vector bundle $E\to X$ such that $\sigma$ generates $\J_V$, then there is
an embedding $N_V\hookrightarrow E|_V$; see, e.g., \cite[Lemma~7.3]{aeswy1}.

\begin{prop}\label{lastbil}
Assume that the normal bundle $N_V X\to V$ is equipped with a Hermitian metric.
For any $\hat\mu\in\mathcal{GZ}(X)$ and $k=0,1,2,\ldots$, there is a generalized cycle 
$\widehat S_k(\J_V,\hat\mu)\in\mathcal{GZ}_{\text{dim}\,\hat\mu-k}(X)$
with the following properties. 

\smallskip

\noindent (i) If $U\subset X$ is open and $\sigma$ is a holomorphic section of a Hermitian vector bundle $E\to U$ such that $\sigma$
generates $\J_V$ in $U$ and the embedding $N_V U\hookrightarrow E|_{V\cap U}$ is an embedding of Hermitian vector bundles, then
$\widehat S_k(\J_V,\hat\mu) = M_k^\sigma\wedge\hat\mu$ in $U$.

\smallskip

\noindent (ii) The image of $\widehat S_k(\J_V,\hat\mu)$ in $\mathcal{B}_{\text{dim}\,\hat\mu-k}(X)$ only depends on the image of $\hat\mu$ 
in $\mathcal{B}(X)$;
in particular it is independent of the Hermitian metric on $N_V X\to V$.
\end{prop} 

If $\mu\in\mathcal{B}(X)$ we let $S_k(\J_V,\mu)$  be the image in $\mathcal{B}_{\text{dim}\,\mu-k}(X)$ of 
$\widehat S_k(\J_V,\hat\mu)$, where $\hat\mu\in\mathcal{GZ}(X)$ is any representative of $\mu$.
 
\begin{proof}
The proof is based on some ideas in \cite{HSKpb}.
Let $\hat\mu\in\mathcal{GZ}(X)$ and assume that $\hat\mu=\tau_*\alpha$, where $\tau\colon W\to X$ is a proper holomorphic mapping and 
$\alpha$ is a product of first Chern forms of Hermitian line bundles on $W$. We can assume that $\tau^*\J_V$ is principal and that $W$ is smooth.
Consider the commutative diagram
\begin{equation*}
\xymatrix{
D  \ar[d]^t \ar[r]^\iota & W \ar[d]^\tau\\
V \ar[r]^i & X,
}
\end{equation*}
where $D$ is the divisor of $\tau^*\J_V$. Let $L\to W$ be the line bundle corresponding to $D$; for future reference we recall that $L|_D$ 
is the normal bundle of $D$. We will show below that the Hermitian metric on $N_V X$ induces a metric on $L|_D$. 
%Assume that $L$ is equipped with a Hermitian metric
Let $\omega=\hat c_1(L|_D^*)$ be the first Chern form of the dual bundle. Then 
\begin{equation}\label{bambu}
\tau_*([D]\wedge \omega^{k-1}\wedge\alpha) = i_*t_*\iota^*(\omega^{k-1}\wedge \alpha)
\end{equation} 
is in $\mathcal{GZ}(X)$. This will be our definition of $\widehat{S}_k(\J_V,\hat\mu)$. However, a priori this definition
depends on the representation $\tau_*\alpha$ of $\hat\mu$.

Let us now describe the induced metric on $L|_D$.
Let $\kappa=\text{codim}\, V$. We recall the following ad hoc definition; cf. \cite[Section~7]{aeswy1}. 
A section $\xi$ of $N_V X$ is a choice of a $\kappa$-tuple $\xi(s)$ locally on $V$ for each local holomorphic $\kappa$-tuple $s$ generating $\J_V$
such that $\xi(Ms)=M\xi(s)$ on $V$ for any locally defined holomorphic matrix $M$ invertible in a neighborhood of $V$.

Assume that $U$, $\sigma$, and $E$ are as in (i). Assume also that $s$ is a holomorphic $\kappa$-tuple generating $\J_V$ in
an open set $U'$. In view of the definition of a section of $N_V X$ above, 
if we consider $s$ as a section of the trivial rank $\kappa$ bundle $F\to U'$, then we can identify
$F|_{V\cap U'}$ with $N_V U'$. Notice that this identification induces a Hermitian metric on $F|_{V\cap U'}$; 
we extend it to a Hermitian metric on $F$ in an arbitrary way.
Since both $s$ and $\sigma$ generate $\J_V$ in $U\cap U'$ there is a holomorphic $A\in \text{Hom}(F,E)$ in
$U\cap U'$ such that $\sigma=As$. In $V\cap U\cap U'$, the embedding $N_V U\hookrightarrow E|_{V\cap U}$, which by assumption is an embedding of
Hermitian bundles, is then realized by $A|_{V}$;
cf.\ \cite[Lemma~7.3]{aeswy1}. 

Let $U''=\tau^{-1}(U\cap U')$. In $D\cap U''$ we get that $a:=\tau^*A|_{D\cap U''}$ embeds $\tau^* N_V U$ in $\tau^* E|_D$.
Moreover, we have a similar situation in $\tau^{-1}(U')$ and in $\tau^{-1}(U)$ as we had in $U\cap U'$ since the ideal sheaf $\tau^*\J_V$,
which defines $D$, 
is generated by $\tau^*s$ in $\tau^{-1}(U')$ and $\tau^*\sigma$ in $\tau^{-1}(U)$. 
In the same way as above, since $L|_D$ is the normal bundle of $D$, we thus get embeddings
\begin{equation*}
L|_{D\cap \tau^{-1}(U')}\hookrightarrow \tau^*F|_{D\cap \tau^{-1}(U')}=\tau^* N_V U' \quad \text{and}\quad 
L|_{D\cap \tau^{-1}(U)}\hookrightarrow \tau^*E|_{D\cap \tau^{-1}(U)}.
\end{equation*}
Using the ad hoc definition of a section of a normal bundle it is straightforward to check that the latter embedding restricted to $U''$
is the composition of
\begin{equation*}
L|_{D\cap U''}\hookrightarrow \tau^*F|_{D\cap U''} \stackrel{a}{\hookrightarrow} \tau^*E|_{D\cap U''}.
\end{equation*}
It follows that the metrics induced on $L|_D$ by the embeddings in $\tau^*F$ and $\tau^*E$, respectively, coincide on $L|_{D\cap U''}$. 
In particular, if $\sigma$ is a holomorphic $\kappa$-tuple generating $\J_V$ in $U$, so that $E|_{V\cap U}$ can be identified with $N_V U$, 
it follows that the metric on $N_V X$ induces a metric on 
$L|_D$.

With this metric on $L|_D$, $M_k^\sigma\wedge\hat\mu$ equals the left-hand side of \eqref{bambu} in $U$ by \cite[Eq.\ (5.9)]{aeswy1}. 
In view of \eqref{taket0}, $M_k^\sigma\wedge\hat\mu$ is independent of the representation $\tau_*\alpha$ of $\hat\mu$. It follows that 
\eqref{bambu} is independent of the representation $\tau_*\alpha$ of $\hat\mu$, and we take \eqref{bambu} as our definition of
$\widehat{S}_k(\J_V,\hat\mu)$. Then $M_k^\sigma\wedge\hat\mu=\widehat{S}_k(\J_V,\hat\mu)$ in $U$ and (i) is proved.

We now note that (ii) follows. Indeed, in view of \cite[Section~3]{aeswy1}, the image of \eqref{bambu} in $\mathcal{B}(X)$ 
is $0$ if $\hat\mu$ is $0$ in $\mathcal{B}(X)$ and, moreover, it is independent of 
the Hermitian metric on $L$.
%To prove the proposition it now suffices to show that the Hermitian metric on $N_V X$ induces a metric on $L|_D$ %gives a Chern form $\hat c_1(L^*)$ 
%and that part
%(i) of the proposition holds. Indeed, if we define $\widehat S_k(\J_V,\hat\mu)$ to be \eqref{bambu} and (i) holds, then in view of \eqref{taket0}, 
%$\widehat S_k(\J_V,\hat\mu)$ only depends on $\hat\mu$ and the Hermitian metric on $N_V X$. 
%Moreover, by \cite[Theorem~5.2]{aeswy1},
%if $\hat\mu$ is $0$ in $\mathcal{B}(X)$, then $\widehat S_k(\J_V,\hat\mu)$ is $0$ in $\mathcal{B}(X)$.
\end{proof}

Let $N_\Delta (Y\times Y)\to \Delta$ be the normal bundle.
\begin{df} 
Given $\mu_1,\mu_2\in\B(Y)$, then 
$\mu_1\di_Y\mu_2$ is the unique element in $\B(Y)$ such that
$$
j_*(\mu_1\di_Y\mu_2)= c(N_\Delta (Y\times Y))\w S(\J_\Delta, \mu_1\times \mu_2).
$$
\end{df}

If we identify $\Delta$ and $Y$, then 
$N_\Delta (Y\times Y)$ is isomorphic to $TY$  
so if we identify $S(\J_\Delta, \mu_1\times \mu_2)$ with an element in $\B(Y)$, then
%we have 
\begin{equation}\label{apa1}
\mu_1\di_Y \mu_2=c(TY)\w S(\J_\Delta, \mu_1\times \mu_2).
\end{equation}

 For the proof of Theorem~\ref{main1} we need the following lemma.
 Recall that a coherent ideal sheaf $\J\to X$, with zero set $Z$, on a reduced space $X$ of pure dimension defines a {\it regular embedding} of
 codimension $\kappa$ if $\codim Z=\kappa$ and locally $\J$ is generated by $\kappa$ functions.  Then 
 there is a well-defined normal bundle $N_{\J} X$ over $Z$.
 See, e.g., \cite[Section~7]{aeswy1}.
 
 \begin{lma}\label{krokant}
 Let $X'$ be a reduced space and let $\iota\colon X\to X'$ be  a reduced subspace.  Assume that 
the coherent
 sheaf $\J'\to X'$ defines a regular embedding of codimension $\kappa$ in $X'$, 
 and that $\J=\iota^*\J'$ defines a regular embedding of codimension $\kappa$  in $X$. 
 Then $N_{\J} X=\iota^* N_{\J'}X'$. 
 \end{lma}
 
 Let $Z$ and $Z'$ denote the zero sets of $\J$ and $\J'$, respectively.

 \begin{proof}
By assumption, locally we have a set of generators $s=(s_1,\ldots, s_\kappa)$ for $\J'$. If $s'$ is another
such $\kappa$-tuple, then (on the overlap) there is an invertible holomorphic $\kappa\times\kappa$ matrix
$a(s,s')$ such that $s'=a(s,s') s$. The matrices so obtained form the transition matrices on $Z'$ for the
bundle $N_{\J'}X'$.   Now the lemma follows by noting that $\iota^*s$ and $\iota^*s'$ are minimal sets
of generators for $\J=\iota^*\J'$ and hence
$\iota^* a(s,s')$ are transition matrices for $N_{\J}X\to Z$. 
 \end{proof}

\begin{proof}[Proof of Theorem~\ref{main1}]

It is clear that $\mu_1\di_Y\mu_2$ is $\Z$-bilinear and commutative since $S(\J_\Delta,\mu_1\times\mu_2)$
is, cf.~Section~\ref{sklass}.
Moreover, its Zariski support is contained in $\Delta\cap (|\mu_1|\times|\mu_2|)$ which after
identifying $\Delta$ and $Y$ is precisely $V=|\mu_1|\cap|\mu_2|$.  Thus (i) holds.

It follows from  \eqref{taket} and \eqref{taket12} that
\begin{equation}\label{tosca}
\epsilon_\ell(\mu_1,\mu_2,x)=
\mult_{j(x)} S_k(\J_\Delta, \mu_1\times\mu_2)
\end{equation}
where  $k=\dim\mu_1+\dim\mu_2-\ell$. 
Since  $c(N_\Delta (Y\times Y))=1+\cdots$, where $\cdots$ are smooth forms of positive bidegree,
it follows from the dimension principle that 
\begin{equation}\label{tosca2}
j_*(\mu_1\di_Y\mu_2)_\ell=\big( c(N_\Delta (Y\times Y))\w S(\J_\Delta, \mu_1\times\mu_2)\big)_\ell=
S_k(\J_\Delta, \mu_1\times\mu_2) +\cdots,
\end{equation}
where $\cdots$ are smooth forms of positive degree times generalized cycle classes.  Now (ii) follows
from \eqref{tosca}, \eqref{tosca2} and the comment after \eqref{tosca50}.

\smallskip
We now prove (iii). Assume that $\mu_j$ are cycles that intersect properly.  Then $\Delta$ intersects $X:=\mu_1\times\mu_2$
properly so that if $\iota\colon X\to Y\times Y$, then $\J:=\iota^*\J_\Delta$ defines a regular embedding in
$X$.  In view of Lemma~\ref{krokant},  \eqref{trottoar} and \eqref{bus} we have, using the notation 
$S(\J,\1_X)=S(\J,X)$, 
\begin{equation}\label{foton}
j_*(\mu_1\di_Y\mu_2)=c(N_{\J_\Delta} (Y\times Y))\w S(\J_\Delta,X)=\iota_*\big(c(N_\J X)\w S(\J,X)\big)=
\iota_*[Z_\J],
\end{equation}
where the last equality is precisely \cite[Theorem~1.4]{aeswy1}. Here $[Z_\J]$ is the
Lelong current of the fundamental cycle associated with $\J$. Its Zariski support is precisely $Z$ but
there is a certain multiplicity of each irreducible component of $Z$. 
Since the right hand side of \eqref{foton} has the expected dimension
$\rho$, cf.~\eqref{rho}, \eqref{foton} implies that 
\begin{equation}\label{foton2}
j_*(\mu_1\di_Y\mu_2)_\rho=%c(N_{\J_\Delta} (Y\times Y))\w S(\J_\Delta,X)=\iota_*\big(c(N_\J X)\w S(\J,X)\big)=
\iota_*[Z_\J].
\end{equation}
Furthermore,
$(\mu_1\di_{Y}\mu_2)_\rho=\mu_1\cdot_{\B(Y)}\mu_2$,
where $\mu_1\cdot_{\B(Y)}\mu_2$ is the product from \cite[Section~5]{aeswy2}.
%Moreover, 
%\begin{equation}\label{polyp}
%(\mu_1\di_{Y}\mu_2)_\rho=\mu_1\cdot_{\B(Y)}\mu_2,
%\end{equation}
%where $\mu_1\cdot_{\B(Y)}\mu_2$ is the product from \cite[Section~5]{aeswy2}.
Since $\mu_j$ intersect properly, this product is equal to $\mu_1\cdot_Y\mu_2$
by \cite[Proposition~5.8~(i)]{aeswy2}. Thus 
\begin{equation}\label{polyp}
\mu_1\cdot_Y\mu_2=(\mu_1\di_Y\mu_2)_\rho.
\end{equation}
Now (iii) follows from \eqref{foton}, \eqref{foton2} and \eqref{polyp}.

\smallskip
We know from \cite[Theorem~1.3]{aeswy2}  that the image of $\mu_1\cdot_{\B(Y)}\mu_2$ 
coincides with the image
of the Chow class $\mu_1\cdot_Y\mu_2$. Thus (iv) follows from \eqref{polyp}. 
This concludes the proof.
\end{proof}

For future reference we include the following simple proposition.

\begin{prop}\label{koks}
Assume that we have an embedding  $i\colon Y\to Y'$ where $Y,Y'$ are smooth. Then 
\begin{equation}\label{anka}
i_*\mu_1\di_{Y'}i_*\mu_2=i_*\big(i^*c(N_{i(Y)}Y')\w \mu_1\di_{Y}\mu_2\big). 
\end{equation}
\end{prop}

\begin{proof}
Notice now that if $\Delta$ and $\Delta'$ are the diagonals in $Y\times Y$ and $Y'\times Y'$,
respectively,  then %xx dsince smooth,  
$$
\J_{\Delta}=(i\times i)^* \J_{\Delta'}.
$$
We claim that 
\begin{equation}\label{listbil}
(i\times i)_* S(\J_{\Delta}, \mu_1\times \mu_2)= S(\J_{\Delta'}, i_*\mu_1\times i_*\mu_2).
\end{equation}
If $\J_{\Delta'}$ is generated by a holomorphic section of a Hermitian vector bundle over $Y'\times Y'$, then this follows from
\eqref{trottoar} since 
$
(i\times i)_*(\mu_1\times \mu_2)=i_*\mu_1\times i_*\mu_2
$.
The general case follows since 
\eqref{trottoar}, with $f$ replaced by $i\times i$, $\J=\J_{\Delta'}$, and $\mu=\mu_1\times\mu_2$, still holds
in view of Proposition~\ref{lastbil}.
% Since 
% $
% (i\times i)_*(\mu_1\times \mu_2)=i_*\mu_1\times i_*\mu_2
% $
% it follows from \eqref{trottoar}  that
% $$
% (i\times i)_* S(\J_{\Delta}, \mu_1\times \mu_2)= S(\J_{\Delta'}, i_*\mu_1\times i_*\mu_2).
% $$

Identifying $\B(\Delta)$ and $\B(\Delta')$ with $\B(Y)$ and $\B(Y')$, respectively, we have 
\begin{equation}\label{ponny}
i_*S(\J_{\Delta}, \mu_1\times \mu_2)=S(\J_{\Delta'}, i_*\mu_1\times i_*\mu_2).
\end{equation}
By \eqref{apa1},
$$
S(\J_{\Delta}, \mu_1\times \mu_2)=c(TY)^{-1}\w \mu_1\di_Y\mu_2,
$$
and  multiplying by $i^* c(TY')=c(TY)\w i^*c(N_{i(Y)}Y')$ we get
\begin{equation}\label{ponny3}
i^* c(TY')\w S(\J_{\Delta}, \mu_1\times \mu_2)=i^* c(N_{i(Y)}Y')\w \mu_1\di_{Y}\mu_2.
\end{equation}
Therefore, by \eqref{apa1}, \eqref{ponny}, and \eqref{ponny3}, 
\begin{multline*}
i_*\mu_1\di_{Y'}i_*\mu_2=c(TY')\w S(\J_{\Delta'}, i_*\mu_1\times i_*\mu_2)=
i_*\big( i^*(c(TY') \w S(\J_{\Delta}, \mu_1\times \mu_2)\big)=\\
%i_*\big(c(N_{Y}Y')\w \mu_1\di_{Y}\mu_2\big)=
i_*\big(i^*c(N_{i(Y)}Y')\w \mu_1\di_{Y}\mu_2\big).
\end{multline*}
%
%Applying $i_*$ and using  \eqref{ponny} we get \eqref{anka}.
%%%
\end{proof}

\section{The $\di$ and $\bullet$-products on $\Pk^n$}\label{kongas}
 
We first recall the definition of the $\bullet$-product on $\Pk^n$.  Let $\eta_0,\ldots,\eta_n$ be sections of
$L=\Ok(1)_{\Pk^{2n+1}}$ that define the join diagonal $\Delta_J$
in   $ \Pk^{2n+1}$, cf.~\cite[Section~6]{aeswy2}.   Let $\J_J$ be the sheaf that defines $\Delta_J$.
Let $\eta'_k$ be holomorphic functions that represent $\eta_k$ in a given local frame for
$L$.  Then 
$$
dd^c\log|\eta|^2_\circ:=dd^c\log(|\eta_0'|^2+\cdots +|\eta_n'|^2)
$$
is a well-defined global current. 
For $\mu\in\B(\Pk^{2n+1})$ we define $V_k(\Delta_J,L, \mu)$ %= 
as the classes in $\B(\Pk^{2n+1})$ defined by
\begin{equation}\label{pork}
M_k^{L,\eta}\w\hat\mu:=\1_{\Delta_J}(dd^c\log|\eta|^2_\circ)^k\w\hat\mu,  \quad k=0,\ldots, n+1,
\end{equation}
where $\hat\mu$ is a generalized cycle that represents  $\mu$.
It is proved in \cite[Section~4]{aeswy2} that the Monge-Amp\`ere products in
\eqref{pork} are well-defined and that 
$V_k(\Delta_J, L, \mu)$
is independent of the choice of  representative 
$\hat\mu$ and sections $\eta_1,\ldots\eta_n$ defining $\Delta_J$. 
If $k>n+1$ in \eqref{pork}, then  %the Monge-Amp\`ere product in \eqref{pork} vanishes,
%so that 
$V_k(\Delta_J,L,\mu)=0$.  
Let $V(\Delta_J,L,\mu)=V_1(\Delta_J,L,\mu)+V_{2}(\Delta_J,L,\mu)+\cdots.$

If $\mu_1,\mu_2\in\B(\Pk^n)$, then there is a natural class $\mu_1\times_J\mu_2$, see \cite[Section~6]{aeswy2},
in  $\B(\Pk^{2n+1})$, generalizing the usual join when $\mu_1,\mu_2$ are cycles, 
and  $\dim (\mu_1\times_J\mu_2)=\dim\mu_1+\dim\mu_2+1$ if $\mu_1$ and $\mu_2$ have pure dimensions. 
Let $j\colon \Pk^n\to \Pk^{2n+1}$, $[x]\mapsto ([x],[x])$, be the natural  parametrization of $\Delta_J$.

\smallskip 
For $\mu_1,\mu_2\in\B(\Pk^n)$ of pure dimensions  
\begin{equation}\label{tagel}
\mu_1\bullet\mu_2=\sum_{\ell=0}^n (\mu_1\bullet\mu_2)_\ell
\end{equation}
is the class in $\B(\Pk^n)$ defined by 
\begin{equation}\label{taket3}
j_*(\mu_1\bullet\mu_2)_\ell= V_{k(\ell)} (\Delta_J, L, \mu_1\times_J\mu_2),
\end{equation}
where
\begin{equation}\label{stork}
k(\ell)=\dim\mu_1+\dim\mu_2+1 - \ell.
\end{equation}
%%

%Since $k\le n+1$,  $\ell\ge\rho$, and hence the sum actually

\smallskip

Let $\omega_{\Pk^n}$ be the first Chern class of $\Ok(1)\to\Pk^n$, for instance represented by the Fubini-Study
metric form.  If $i\colon W\to \Pk^n$ is a linear subspace, with the induced metric, then
$\omega_W=i^*\omega_{\Pk^n}$.  We will often write $\omega$ without subscript.  
Recall that 
$
\hat c(T\P^n)=(1+\omega)^{n+1}.
$

%In \cite{aeswy2} was introduced a product ablablabla.   From oxet, the proof of noxet, we have
%that

\begin{prop}
Let $\di=\di_{\Pk^n}$
and let
$$
\rho=\dim\mu_1+\dim\mu_2-n, \quad V=|\mu_1|\cap|\mu_2|. 
$$
We have the relations
\begin{equation}\label{potatis}
\mu_1\di \mu_2= 
\sum_{\ell=\max(0,\rho)}^{\dim V}
(1+\omega)^{\ell-\rho}
 (\mu_1\bullet\mu_2)_\ell.
\end{equation}
and  
\begin{equation}\label{potatis2}
\mu_1\bullet \mu_2= 
\sum_{k=0}^{\dim V}
(1-\omega)^{k-\rho}
 (\mu_1\di\mu_2)_k.
\end{equation}
\end{prop}

Since $k\le n+1$,  $\ell\ge \rho$. Moreover, each term in the sum \eqref{tagel} has support on $V$.  
Hence the sum runs from  $\max(\rho,0)$ to $\dim V$.

\begin{proof}
With the notation in \cite[Section~7]{aeswy2} we have that
$\mu_1\di\mu_2=i^{!!}(\mu_1\times\mu_2)$.  
It follows from \cite[Proposition~7.1]{aeswy2} that
$$
j_*(\mu_1\di\mu_2)=c(N_{\J_J}\Pk^{2n+1})\w S(\J_J,\mu_1\times_J\mu_2)=
(1+\omega)^{n+1}\w S(\J_J,\mu_1\times_J\mu_2).
$$
By the second van~Gastel type equality in \cite[Corollary~9.9]{aeswy1} we get % cf.~the proof of \cite[Theorem~2.1]{aeswy2},
\begin{equation}\label{skal}
j_*(\mu_1\di\mu_2)=\sum_{k\ge 0}(1+\omega)^{n+1-k} \w V_k(\Delta_J,L,\mu_1\times_J\mu_2).
%=\sum_{k\ge 0}(1+\omega)^{n+1-k}\w
\end{equation}
%Here $\omega$ is 
%$\omega_{\Delta_J}$ is $\omega=\omega_{\Delta_J}$.  
Since $k\le n+1$ it follows that $\ell\ge\rho$, cf.~\eqref{stork},  and since all terms in the sum
have Zariski support on $V$,
each term with $\ell$ larger than $\dim V$ must vanish in view of the dimension principle.
%Since $j^*\omega_{\Delta_J}=\omega_{\Pk^n}$  
Hence \eqref{skal} is precisely \eqref{potatis}. 
 
The equality \eqref{potatis2} follows from \eqref{potatis} and Lemma~\ref{mangel} below. Notice that although the
sum in \eqref{potatis} happens to begin at $\ell=\max(\rho,0)$ it will give rise to terms of lower dimension so 
 \eqref{potatis2} must start at $k=0$.
\end{proof}

\begin{lma}\label{mangel}
Assume that $A=\sum_{\ell\ge 0} A_\ell$ is a graded $\C$-algebra and  
$\omega\colon A\to A$ maps $A_{\ell+1}\to A_\ell$, $\ell\ge 0$,  and $A_0\to 0$. 
Moreover, let $r$ be a fixed integer.
Assume that $a=a_0+a_1+\cdots$, where $a_\ell$ are elements in $A_{\ell}$,  and let
$b_k$ be the elements in $A_{k}$ so that  
\begin{equation}\label{klara}
\sum_{k=0}^m b_k=\sum_{\ell=0}^m(1+\omega)^{\ell+r} a_\ell.
\end{equation}
%where $b_k$ has dimension $k+r$, then
 Then 
 $$
\sum_{\ell=0}^m a_\ell=\sum_{k=0}^m(1-\omega)^{k+r} b_k.
$$
\end{lma}

%Notice that the right hand side in \eqref{klara} has components of degrees running from
%$0$ so that all $b_k$ are well-defined. 
This lemma is probably well-known but we sketch a proof.

\begin{proof} [Sketch of proof]
We can identify $a\in A$ with the $A$-valued meromorphic function 
$$
z\mapsto \hat a(z)=\sum_{\ell\ge 0} z^{\ell+r} a_\ell.
$$
Let 
$$
T_\omega a := \sum_{\ell= 0}^m(1+\omega)^{\ell+r} a_\ell.
$$
Since $(z+\omega)^{\ell+r} a_\ell=(1+\omega/z)^{\ell+r} z^{\ell+r} a_\ell$ it follows that 
$$
\widehat {T_\omega a}(z)= \sum_{\ell= 0}^m (z+\omega)^{\ell+r} a_\ell,
$$
i.e.,  $\widehat {T_\omega a}(z)$ is obtained by formally replacing each occurrence of $z$ in $\hat a(z)$
by $z+\omega$. 
%This can also be considered as a simple instance of operator calculus. 
It is now clear that 
$T_{-\omega}\circ T_\omega=Id$ which proves the lemma.
\end{proof}

Recall, cf.~\eqref{kulting},  that $\mu\in\B_k(\Pk^n)$ has the {\it degree}
\begin{equation}\label{degdef}
\deg\mu:=\int_{\Pk^n}\omega^k\w \mu.
\end{equation}
%where $\omega$ is the first Chern class of $\Ok(1)\to\Pk^N$, for instance represented by the Fubini-Study
%metric form. 
If $\mu=\mu_0+\mu_1+\cdots $, where $\mu_k$ has pure
dimension $k$, then 
%\in\B_k(\Pk^n)$, then 
$\deg \mu:=\deg\mu_0+\deg\mu_1+\cdots$.

\begin{proof}[Proof of \eqref{buske}]
From \eqref{potatis} we have that 
\begin{equation}\label{pompe}
\mu_1\di\mu_2=\sum_{\ell=\max(0,\rho)}^{\dim V}
(1+\omega)^{\ell-\rho}
 (\mu_1\bullet\mu_2)_\ell=
 \sum_{\ell=\max(0,\rho)}^{\dim V}
\sum_{j=0}^{\min(\ell-\rho,\ell)}  
{{\ell-\rho} \choose j}
\omega^j \w (\mu_1\bullet\mu_2)_\ell
\end{equation}
since  $\omega^j \w (\mu_1\bullet\mu_2)_\ell=0$ for degree reasons when
$j>\ell$. 
We get the estimate
\begin{multline*}
\deg (\mu_1\di\mu_2)=\sum_{\ell=\max(0,\rho)}^{\dim V}
\sum_{j=0}^{\min(\ell-\rho,\ell)}  
{{\ell-\rho} \choose j}
\deg (\omega^j \w (\mu_1\bullet\mu_2)_\ell) \le 
 \\
 \sum_{\ell=\max(0,\rho)}^{\dim V} 
 \sum_{j=0}^{\ell-\rho}  
{{\ell-\rho} \choose j}
 \int_{\Pk^n}\omega^{\ell}\w
 (\mu_1\bullet\mu_2)_{\ell} =\\
 \sum_{\ell=\max(0,\rho)}^{\dim V} 
 2^{\ell-\rho}  \int_{\Pk^n}\omega^{\ell}\w
 (\mu_1\bullet\mu_2)_{\ell} 
 \le 2^{\dim V-\rho}\deg (\mu_1\bullet\mu_2).
 \end{multline*}
\end{proof}

In view of the proof we have equality in \eqref{buske} if $\rho\ge 0$ and in addition 
only the term with $\ell=\dim V$ occurs.

\begin{ex}
Let $\mu_1$ and $\mu_2$ be the same $k$-plane $V$ in $\Pk^{n}$.  Then $V\bullet V=V$, 
see \cite[Section~1]{aeswy2}.  Thus only the term corresponding to $\ell=\dim V=k$ occurs in 
\eqref{potatis}. If in addition $\rho\ge 0$, i.e., $2k\ge n$, then each term in the expansion of 
$(1+\omega)^{\ell-\rho}$ gives a 
contribution and therefore, since $\ell=\dim V$,
$$
\deg (V\di V)=2^{\dim V-\rho} \deg(V\bullet V)
$$
so that the estimate  \eqref{buske} is sharp. 
 \end{ex}

\section{The $\bullet_L$-product on a projective manifold $Y$}\label{bl}
We shall now see that if $Y$ is projective and $L\to Y$ is a very ample line bundle,  then  there is 
an associated product $\bl$ with the desired local multiplicities and a B\'ezout inequality 
for effective generalized cycle classes. 

By definition  `very ample'  means that there is an embedding 
\begin{equation}\label{lalla}
i\colon Y\to \Pk^M
\end{equation}
such that $L=i^*\Ok_{\Pk^M}(1)$.  
For $\mu_1,\mu_2\in \B(Y)$, we define  $\mu_1\bl\mu_2$ as the unique
element in $\B(Y)$ such that
\begin{equation}\label{lbullet}
i_*(\mu_1\bl\mu_2)=i_*\mu_1\bullet i_*\mu_2,
\end{equation}
where the right hand side is the $\bullet$-product in $\P^M$. We shall see that $\bl$ only depends
on $L$ and not on the embedding $i$.

Assume that $\Pk^M=\{[z_0,\ldots,z_m]; \ z\in\C^{M+1}\}$. Since then
$z_0,\ldots, z_M$ are global sections of $\Ok(1)\to\Pk^M$ it follows
that $s_k:=i^* z_k$, $k=0,\ldots, M$, are  in $H^0(Y,L)$.
Moreover, $i$ is given by
%given by 
\begin{equation}\label{gogo}
x\mapsto [s_0(x),\ldots,s_M(x)].
\end{equation}
Conversely, if we have $s_0,\ldots, s_M$ in $H^0(Y,L)$ such that \eqref{gogo} defines an embedding,
then $i^*\Ok(1)=L$. In fact, since $g_{\ell k}=z_\ell/z_k$ are transition functions for $\Ok(1)$, 
$s_\ell/s_k=i^*g_{\ell k}$ are transition functions for $i^*\Ok(1)$.  
Let 
$$
N=\dim H^0(Y,L)-1.
$$
Given an embedding \eqref{gogo} let us select a maximal linearly independent subset 
$s_0,\ldots,s_{N'}$ of the $s_k$.  Notice that then $N'\le N$. 
Let $i'\colon Y\to \Pk^{N'}$ be the embedding defined by these sections.
Then, there is a linear subspace $\iota\colon V\to \Pk^M$ such that
$i=\iota\circ i'$.  In view of \cite[Proposition~6.7]{aeswy2}, $i$ and $i'$ give rise to the same
product on $Y$.

\smallskip
Thus we can assume that our embedding \eqref{lalla} is defined by \eqref{gogo}, where
$s_0,\ldots,s_M$ is a linearly independent set in  $H^0(Y,L)$. 
In view of \cite[Example~6.4]{aeswy2} the product $\bl$ only depends on the subspace of  
$H^0(Y,L)$ spanned by the given sections.
\begin{prop}\label{snarstucken}
Assume that we have the embedding \eqref{lalla} and let 
 \begin{equation}\label{hatrho}
 \hat \rho=\dim\mu_1+\dim\mu_2-M.
\end{equation}
Assume that $\mu_1,\mu_2\in\B(Y)$ have pure dimensions and
let $d=\dim\mu_1+\dim\mu_2$ and $V=|\mu_1|\cap|\mu_2|$.
 Then 
\begin{equation}\label{snar1}
\mu_1\di_Y\mu_2=\sum_{\ell=\max(0,\hat\rho)}^{\dim V}
(1+\omega_L)^{\ell-d-1}\w c(TY)\w(\mu_1\bullet_L\mu_2)_\ell
\end{equation}
and
\begin{equation}\label{snar2}
\mu_1\bl\mu_2=\sum_{k=0}^{\dim V}
(1-\omega_L)^{k-d-1}(c(TY)^{-1} \w \mu_1\di_Y\mu_2)_k.
\end{equation}
\end{prop}

Since the right-hand side of \eqref{snar2} only depends on $L$,  this holds for $\mu_1\bl\mu_2$
as well.

\begin{proof} 
%Fix an embedding \eqref{lalla}.  
By Proposition~\ref{koks} and \eqref{potatis} we have
$$
i_*\big(i^*c(N_{i(Y)}\Pk^M)   \mu_1\di_Y\mu_2\big)=\sum_{\ell=\max(0,\hat\rho)}^{\dim V}(1+\omega)^{\ell-\hat\rho}
(i_*\mu_1\bullet i_*\mu_2)_\ell.
$$
Notice  that since 
$$
N_{i(Y)}\Pk^M=T\Pk^M/T(i(Y)),
$$
and $i^*\Ok(1)=L$, so that $i^*\omega=\omega_L$, we have that 
$$
i^*c(N_{i(Y)}\Pk^M)= i^*\big( c(T\Pk^M) c(T(i(Y))^{-1}\big)=(1+\omega_L)^{M+1} c(TY)^{-1}
$$
on $Y$. 
Thus
$$
(1+\omega_L)^{M+1} c(TY)^{-1}\w  \mu_1\di_Y\mu_2=
\sum_{\ell=\max(0,\hat\rho)}^{\dim V}(1+\omega_L)^{\ell-\hat\rho}
(\mu_1\bl \mu_2)_\ell
$$
which is the same as \eqref{snar1}.  Now  \eqref{snar2} follows by Lemma~\ref{mangel}.
\end{proof}

%\noindent 
Notice that there may occur negative powers of $1-\omega_L$ and $1+\omega_L$ in the sums
\eqref{snar1} and \eqref{snar2}.

 \smallskip

%\noindent 
Recall that if  $\mu\in\B_k(Y)$, then, cf.~\eqref{kulting}, 
% $\mu\in\B_k(Y)$ we define % for some embedding en 
$$
\deg_L\mu=\int_Y \mu\w \omega_L^k.
$$
%for and extend to $\B(Y)$ by linearity. 

 \begin{proof}[Proof of Theorem~\ref{main2}]
 Parts (i) and (ii) follow from Theorem~\ref{main1} and \eqref{snar2}.
 %%%
Part (iii) follows from the corresponding statement for $\bullet=\bullet_{\Pk^M}$ and
\eqref{tosca50}.
Alternatively, it follows
from \eqref{snar2} and Theorem~\ref{main1}~(ii). %he analogous statement for $\di_Y$, cf.~Section~\ref{multi}.

Part (iv) follows from the analogous statement for $\bullet$ on $\Pk^M$. In fact,
first notice that $\mu$ is effective if and only if $i_*\mu$ is.  Then observe that
if $\mu$ has pure dimension $k$, then
$$
\deg_L\mu=\int_Y \mu\w\omega_L^k=\int_{\Pk^M} i_*\mu\w\omega^k=\deg_{\Pk^M} i_*\mu.
$$
If $\mu_1$ and $\mu_2$  are effective, then $i_*(\mu_1\bl\mu_2)=i_*\mu_1\bullet i_*\mu_2$ is effective,
see \cite[Theorem~1.1]{aeswy2}, 
and hence $\mu_1\bl\mu_2$ is. 
From \cite[Theorem~1.1]{aeswy2} we thus have that 
$$
\deg_L (\mu_1\bl\mu_2)= \deg_{\Pk^M} (i_*\mu_1\bullet i_*\mu_2)\le \deg_{\Pk^M} i_*\mu_1\cdot \deg_{\Pk^M} i_*\mu_2=
\deg_L\mu_1\cdot \deg_L\mu_2
$$
with equality if $\hat\rho\ge 0$.  

 \smallskip

Let us now consider (v).  If $\mu_1$ and $\mu_2$ are cycles that intersect properly on $Y$,
then by Theorem~\ref{main1}, 
\begin{equation}\label{sus}
\mu_1\di_Y\mu_2=(\mu_1\di_Y\mu_2)_\rho=\mu_1\cdot_Y\mu_2,
\end{equation}
where  $\rho=\dim V=\dim\mu_1+\dim\mu_2-n$.
From \eqref{snar2} we have   
$$
\mu_1\bl\mu_2=
\sum_{k=0}^{\rho}
(1-\omega_L)^{k-d-1}(c(TY)^{-1} \mu_1\di_Y\mu_2)_k.
$$
Now $k=\rho$ together with the term $1$ from $(1-\omega_L)^{k-d-1}$ 
gives us  $\mu_1\cdot_Y\mu_2$,  cf.~\eqref{sus}.  All other terms
from $(1-\omega_L)^{k-d-1}$, or for $k<\rho$, will give contributions of strictly lower
dimension, and they have vanishing multiplicities, see Section~\ref{multi}.
 %that is, expressed with the Segre class $s(TY)=c(TY)^{-1}$, 
%\begin{equation}\label{plym}
%\mu_1\bl\mu_2=
 %\mu_1\cdot_Y\mu_2
 %\sum_{k=0}^{\rho}(1-\omega_L)^{k-d-1}s(TY)
%\end{equation}
%%
%Thus the leading term is $\mu_1\cdot_Y\mu_2$ and the other terms have lower
%dimension.
 %
 \end{proof}

We have the following consequence of the proof.

\begin{prop} \label{sugga}
Let $M+1$ be the minimal dimension of a subspace $W$ of $H^0(Y,L)$ such that
\eqref{gogo} is an embedding if $s_0,\ldots,s_M$ is a basis for $W$. If 
$\mu_1,\mu_2\in \B(Y)$ have pure dimensions and
$\hat\rho=\dim\mu_1+\dim \mu_2-M\ge 0$, then 
$$
\deg_L(\mu_1\bl\mu_2)=\deg_L \mu_1\cdot \deg_L \mu_2.
 $$
 \end{prop}

%Notice that if $i$ is a linear embedding of $Y=\P^n$ in $\P^N$, then
%$$
%c(N_Y\Pr^N)=(1+\omega)^{N-n}
%$$
%and hence we get back the formula balblablaa.

\section{Some further properties}\label{further}

In this section we still assume that $Y$ is a projective manifold.
% although some statements
%for $\di$ hold in an arbitrary complex manifold $Y$. 
%
Assume that  $\mu_0,\mu_1\in\B(Y)$ and that $\gamma$ is a smooth (closed) form in an open
subset $U\subset Y$. We say that
$\mu_1=\gamma\w\mu_0$ in  $U$ if  there are generalized cycles  
$\mu_0'$ and $\mu_1'$ representing $\mu_0$ and $\mu_1$, respectively, 
such that $\mu_1'=\gamma\w\mu_0'$ in $U$. 
We have the following version of Proposition~8.4 in \cite{aeswy2}. 

\begin{prop}\label{forsta}
Assume that $\mu_0,\mu_1,\mu_2\in\B(Y)$,  $\gamma$ is smooth in the open set
$U\subset Y$,  and $\mu_1=\gamma\wedge \mu_0$ in $U$.
Then 
%Then 
\begin{equation}\label{mars1}
\mu_1\di_Y \mu_2 = \gamma \wedge (\mu_0 \di_Y \mu_2)
\end{equation}
in $U$.
If $L\to Y$ is a very ample line bundle, then 
\begin{equation}\label{mars2}
\mu_1\bl \mu_2 = \gamma \wedge (\mu_0 \bl \mu_2)
\end{equation}
in $U$.
\end{prop}

\begin{proof}
Fix suitable representatives $\mu_0',\mu_1',\mu_2'$ in $\GZ(Y)$ and a section $\eta$ that defines the diagonal $\Delta$
in $Y\times Y$.  Moreover, let $\hat c(N_\Delta (Y\times Y))$ be a fixed representative of the Chern class
$c(N_\Delta (Y\times Y))$.  As usual, let $j\colon Y\to Y\times Y$ be the natural parametrization of 
$\Delta$.
Then  $j_*(\mu_1\di_Y\mu_2)$ is represented, cf.~Section~\ref{sklass},  by the generalized cycle
\begin{equation}\label{tupp1}
j_*(\mu'_1\di_Y\mu'_2):=\hat c(N_\Delta (Y\times Y))\w M^\eta \w(\mu_1'\times\mu_2').
\end{equation}
By assumption $\mu_1'=\gamma\w\mu_0'$ in $U$.  Thus 
$\mu_1'\times\mu_2'=(\gamma\times 1)\w (\mu_0'\times\mu_2')$ in $U\times U$.   In view of 
\cite[Example~2.7]{aeswy2} therefore
\begin{equation}\label{tupp2}
M^\eta\w (\mu_1'\times\mu_2')=(\gamma\times 1)\w M^\eta \w(\mu_0'\times\mu_2')
\end{equation}
in $U\times U$.   
By \eqref{tupp1} and \eqref{tupp2}, 
\begin{multline*}
j_* (\gamma\w (\mu_0'\di_Y\mu_1'))=j_*( j^*(\gamma\times 1)\w (\mu_0'\di_Y\mu_1'))=\\
(\gamma\times 1)\w j_*(\mu_0'\di_Y\mu_1')= (\gamma\times 1)\w 
\hat c(N_\Delta (Y\times Y))\w M^\eta \w(\mu_0'\times\mu_2')=\\
\hat c(N_\Delta (Y\times Y))\w M^\eta \w(\mu_1'\times\mu_2')=j_*(\mu_1'\di_Y\mu_2')
\end{multline*}
in $U\times U$. Now \eqref{mars1} follows.

\smallskip
Assume that $\gamma$ has pure degree $\nu$. %Consider \eqref{snar2}. 
Then by \eqref{mars1}, 
\[
(c(TY)^{-1} \mu_1 \di_Y \mu_2)_k=
(\gamma \wedge c(TY)^{-1} \mu_0 \di_Y \mu_2)_k=
\gamma\wedge (c(TY)^{-1} \mu_0 \di_Y \mu_2)_{k+\nu}.
\]
Let $d$ %and $\rho$ 
be as in Proposition~\ref{snarstucken} and let
$\tilde d=\dim \mu_0 +\dim \mu_2$; then $d=\tilde d-\nu$. 
% and $\tilde \rho=\tilde d-N$. Then
%$d=\tilde d-g$ and $\rho=\tilde \rho-g$. 
By \eqref{snar2},
\begin{multline*}
\mu_1\bl \mu_2
=
\sum_{k\geq 0}
(1-\omega_L)^{k-d-1}(c(TY)^{-1} \mu_1\di_Y\mu_2)_k
=\\
\gamma\wedge \sum_{k\geq 0}
(1-\omega_L)^{k+\nu-\tilde d-1}(c(TY)^{-1} \mu_0\di_Y\mu_2)_{k+\nu}
=\\
\gamma\wedge \sum_{r\geq \nu}
(1-\omega_L)^{r-\tilde d-1}(c(TY)^{-1} \mu_0\di_Y\mu_2)_{r}
=\\
\gamma\wedge \sum_{r\geq 0}
(1-\omega_L)^{r-\tilde d-1}(c(TY)^{-1} \mu_0\di_Y\mu_2)_{r}
=
\gamma \wedge (\mu_0\bl \mu_2),
\end{multline*}
since the terms with $r<\nu$ in the last sum vanish when multiplied by $\gamma$.  
\end{proof}

We have the following version of Proposition 8.3
in \cite{aeswy2}.

\begin{prop}\label{tredje}
If $\mu\in\B(Y)$, then 
\begin{equation}\label{borr}
\mu \di_Y \1_Y=\mu.
\end{equation}
If $a$ is a point in $Y$, then  
\begin{equation}\label{fjarde}
\mu\di_Y \{a\} = \mult_a\mu \cdot [a]
\end{equation}
and 
\begin{equation}\label{femte}
\mu\bl \{a\} = \mult_a\mu \cdot [a].
\end{equation}
\end{prop}

\begin{proof}
We can assume that $\mu=\tau_*\alpha$, where $\tau\colon W\to Y$ is proper and $\alpha$ is a product
of components of Chern forms. If $T=\tau\times 1\colon W\times Y\to Y\times Y$, then
$T^*\J_\Delta$ is a regular embedding in $W\times Y$ since this sheaf defines the graph $G$ of
$\tau$ in $W\times Y$.   
Notice that since $N_G(W\times Y)\simeq T|_{G}^*N_\Delta(Y\times Y)$,
\begin{equation}\label{punk1}
c(N_G(W\times Y)) = T|_{G}^*c(N_\Delta(Y\times Y)).
\end{equation} 
Moreover, by \eqref{trottoar} we have
\begin{equation}\label{punk2}
S(\J_\Delta,\mu\otimes 1)=T_* S(T^*\J_\Delta, \alpha\otimes 1),
\end{equation}
and by \cite[Proposition~1.4]{aeswy1}, 
\begin{equation}\label{punk3}
c(N_G(W\times Y))\w S(T^*\J_\Delta,1_{W\times Y})=[G].
\end{equation}
By \cite[Proposition~5.6]{aeswy1} we have that
$S(T^*\J_\Delta, \alpha\otimes 1)=(\alpha\otimes 1)\w S(T^*\J_\Delta, 1_{W\times Y})$.
Together with \eqref{punk1}, \eqref{punk2}, and \eqref{punk3}  we have % in view of 
\begin{multline*}
j_*(\mu\di_Y \1_Y)=c(N_\Delta (Y\times Y))\w S(\J_\Delta,\mu\otimes 1)=
T_*\Big(c(N_G (W\times Y))\w S(T^*\J_\Delta,\alpha\otimes 1)\Big)=\\
T_*\Big((\alpha\otimes 1)\w c(N_G(W\times Y))\w S(T^*\J_\Delta,1_{W\times Y})\Big)=
T_*\big((\alpha\otimes 1)\w [G]\big). %=[\Delta]\w(\mu\times \1_Y)
\end{multline*}
Let $\tilde\tau\colon W\to W\times Y$, $\tilde\tau(w)=(w,\tau(w))$, be the graph embedding.
Then $(\alpha\otimes 1)\w [G]=\tilde\tau_*\alpha$ and $T\circ\tilde\tau=j\circ \tau$. Hence
$$
j_*(\mu\di_Y \1_Y)=T_*\big((\alpha\otimes 1)\w [G]\big)=T_*\tilde\tau_*\alpha=j_*\tau_*\alpha=j_*\mu,
$$
which means that \eqref{borr} holds.

\smallskip
The last two equalities can be verified in several ways. 
%Let us choose coordinates $x,y$ on $Y\times Y$, such that $a$ is the origin. 
%First, by a linear change of coordinates we may
%assume that $a=0$. 
Notice that we can choose a \nbh $U\subset Y$ of $a$ and coordinates $(x,y)$ in $U\times U$.
Then $S(\J_\Delta, \mu\times\{a\})$, restricted to $U\times U$, is represented by 
\begin{equation}\label{pudel}
M^{x-y}\wedge (\mu \otimes [a])=
 (M^{x-a}\wedge \mu)\otimes  [a]=\mult_a\mu \cdot [a]\otimes [a]=\mult_a\mu\cdot j_*[a],
\end{equation}
where %$\xi_a$ generates the maximal ideal at $a$ and
the second equality follows from, e.g., \cite[Equation~(4.5)]{aeswy2}. 
Hence, $S(\J_\Delta, \mu\times\{a\})=S_0(\J_\Delta, \mu\times\{a\})$ and so
$$
j_*(\mu\di_Y \{a\})=c(N_\Delta (U\times U)\w S(\J_\Delta, \mu\times\{a\})=S_0(\J_\Delta, \mu\times\{a\}).
$$
Now \eqref{fjarde} follows from \eqref{pudel}.
Since $\mu\di_Y\{a\}$ has dimension $0$ it follows from \eqref{snar2}
that $\mu\bl \{a\}=\mu \di_Y \{a\}$ and thus \eqref{femte} follows. 

Alternatively \eqref{femte} follows from the definition \eqref{lbullet} and the analogous
statement for $\bullet=\bullet_{\Pk^M}$. Then \eqref{fjarde} follows from \eqref{snar2}
as above.
\end{proof}

From \eqref{borr} and \eqref{snar2} it follows that if $\mu$ has pure dimension, then
\begin{equation}\label{tupplur}
\mu \bl \1_Y= 
\sum_{k=0}^{\dim \mu}
(1-\omega_L)^{k-n-\dim \mu -1}(c(TY)^{-1}\wedge \mu)_k,
\end{equation}
since $\dim V=\dim\mu$.  

%cf.\ Examples BLA (de exempel med sjalvsnitt) below. 

\smallskip

%The formula \eqref{kur} can also be written in the more explicit fashion
 %\begin{equation}\label{kur2}
% \mu_1\di_{Y}\mu_2=\frac{1}{c(N_Y\Pr^N)}\sum_{j\ge 0} (1+\omega)^{N+1-j}\w M^{L,\eta}_j (\mu_1\times_J \mu_2).
 %\end{equation}
% 

Let us now mention  a possible way  to express our products as limits of smooth forms times 
$\mu_1\times\mu_2$. It follows from \cite[Proposition~5.7]{aeswy1} that the representatives
$M^\sigma_k\w \hat\mu$ of $S_k(\J,\mu)$, cf.~\eqref{taket0}, can be computed by the formula
\begin{equation}\label{pommac}
M^\sigma_k\w \hat\mu=
 \lim_{\epsilon\to 0} \frac{\epsilon (dd^c|\sigma|^2)^k}
{ (|\sigma|^2+\epsilon)^{k+1}}\w\hat\mu, \quad k=0,1,\ldots.
\end{equation}
In particular we have
 
\begin{prop}
If $\mu_1,\mu_2\in\B(Y)$ are represented by $\hat\mu_1,\hat\mu_2$, 
and $\eta$ is a section of a Hermitian bundle $E\to Y\times Y$ that defines $\J_\Delta$, then
$j_*(\mu_1\di_Y\mu_2)$ is represented by the limits
\begin{equation}\label{snir}
c(N_\Delta (Y\times Y))\w \sum_{k=0}^{\dim V} \lim_{\epsilon\to 0} \frac{\epsilon (dd^c|\eta|^2)^k}
{ (|\eta|^2+\epsilon)^{k+1}}\w (\hat\mu_1\times\hat\mu_2).
\end{equation}
\end{prop}

One gets a formula for $\mu_1\di_Y\mu_2$ by taking $\pi_*$ of \eqref{snir}, where
$\pi \colon Y\times Y\to Y$ is the projection onto the first (or the second) factor, since $\pi\circ j=Id_Y$.
One can get similar formulas for $\mu_1\bl\mu_2$ by combining \eqref{snir} and \eqref{snar2}.

\section{Examples}\label{examples}

We first recall the so-called Segre embedding
$$
i\colon \Pk^m_x\times\Pk^n_y\to \Pk^{(m+1)(n+1)-1}.
$$
%$i\colon  \Pk^m_x\times\Pk^n_y\to \Pk^{(m+1)(n+1)-1}$.
%
 % of $Y= \Pk^m_x\times\Pk^n_y$ in 
Let $Y= \Pk^m_x\times\Pk^n_y$ and consider the line bundle $L=\Ok(1)_{\Pk^m_x}\otimes\Ok(1)_{\Pk^n_y}$.
The set of sections
$\{x_jy_k; \ \  0\le j\le m, \ 0\le k\le n\}$ is a basis for $H^0(Y,L)$ and $i$ is the 
associated embedding, cf.~Section~\ref{bl},  
$$
%i\colon Y\to \Pk^{(m+1)(n+1)-1}, \quad
([x],[y])\mapsto [x_0 y_0: \ldots: x_0y_n: \ldots: x_my_0: \ldots : x_my_n].
$$
%called the Segre embedding.
Notice that 
\begin{equation}\label{por}
i^*\omega=\omega_L=\omega_x+\omega_y,
\end{equation}
where $\omega$, $\omega_x$ and $\omega_y$ are  (representatives of) 
$c_1(\Ok(1)_{\Pk^{(m+1)(n+1)-1}})$, 
$c_1(\Ok(1)_{\Pk^m_x})$, 
and $c_1(\Ok(1)_{\Pk^n_y})$, 
%Fubini-Study forms on $\Pk^m_x$ and $\Pk^n_y$,
respectively.
%
%\end{ex}

\smallskip

We now consider the self-intersection of an exceptional divisor.

\begin{ex}
Consider the blowup $Y=Bl_p\Pk^2$ of $\Pk^2$ at the point $p=[1:0:0]$, and let
both $\mu_1$ and $\mu_2$ be the exceptional divisor $E$.  We have the embedding
$$
j\colon Y\to Y':= \Pk^2_x\times \Pk^1_y
$$
so that $j(Y)=\{f=0\}$, where $f$ is the section $f=x_1y_1-x_2y_0$ of $\Ok(1)_{\Pk^2_x}\otimes\Ok(1)_{\Pk^1_y}$.
Let $L\to Y$ be the pullback to $Y$ of this line bundle, which we for simplicity denote in the same way
so that
$$
\omega_L=\omega_x+\omega_y.
$$
We now compose with the Segre embedding
$$
\sigma\colon  \Pk^2_x\times \Pk^1_y\to \Pk^5_z,
\quad
([x_0:x_1:x_2],[y_0:y_1])\mapsto [x_0y_0:x_0y_1:x_1y_0:x_1y_1: x_2y_0:x_2y_1]
$$
and get the embedding $i=\sigma\circ j\colon Y\to \Pk^5$.  
%Since  %$i^*\Ok(1)_{\Pk^5}=
%$L=\Ok(1)_{\Pk^2_x}\times\Ok(1)_{\Pk^1_y}$,
%therefore
%$$
%\omega_L=\omega_x+\omega_y.
%$$
%
%
We claim that
\begin{equation}\label{solo1}
E\bl E=E
\end{equation}
and 
\begin{equation}\label{solo2}
E\di_Y E=E-\omega_L\w E.
\end{equation}
In fact,  the image of $E$ in $\Pk^2\times \Pk^1$ is $\{[1:0:0]\}\times \Pk^1_y$ so the image in
$\Pk^5$ is the line $\{[y_0:y_1: 0:0:0:0]\}$. Therefore  $i_*E\bullet i_*E=i_*E$, see
the remark after \cite[Theorem~1.1]{aeswy2},
 and thus \eqref{solo1} holds. 
%
%\smallskip
Next we compute $E\di_Y E$.  In view of \eqref{solo1} only the term with $\ell=1$ occurs in 
\eqref{snar1} and since $\ell-d-1=1-2-1=-2$ we have 
$$
E\di_Y E=(1+\omega_L)^{-2}\w c(TY)\w E.
$$
Now, cf.~\cite[Eq (7.5)]{aeswy1},  $c(N_Y Y')=c(L)=1+\omega_x+\omega_y$ and thus 
$$
c(TY)= c(TY')/c(N_Y Y')=(1+\omega_x)^3(1+\omega_y)^2 /( 1+\omega_x+\omega_y).
$$
Since $\omega_x=0$ on  $E$ we have, cf.~\eqref{por}, 
$$
E\di_Y E=\frac{1}{1+\omega_y}\w E=(1-\omega_y)\w E=(1-\omega_L)\w E.
$$
%and thus \eqref{solo2} follows.
\end{ex}

Let us next look at an example where $Y$ is embedded into $\Pk^M$ for a
minimal $M$, and where the terms  $\cdots$ of lower dimension in \eqref{studsmatta}  do not vanish.

\begin{ex}
Let $Y=\Pk^1_x\times \Pk^1_y$ and let 
\[
i: Y\to \Pk^3, \quad
\big ([x_0:x_1], [y_0:y_1]\big ) \mapsto [x_0 y_0 : x_0y_1:x_1y_0:
x_1y_1],
\]
be the Segre embedding.  
Note that 
$Y\cdot_YY=Y$
since it is a proper intersection. It
follows from Theorem~\ref{main1} that 
\begin{equation}\label{butter}
Y\di_YY=Y.
\end{equation}
We want to compute $Y\bl Y$. Since $\omega_L=i^* \omega = \omega_x+\omega_y$, 
cf.~\eqref{por}, it follows that 
\[
\omega_L^2 = (\omega_x + \omega_y)^2 = 2 \omega_x\w\omega_y
\]
and thus
\[
\deg_L Y = \int_Y \omega_L^2 = \int_{\Pk^1_x\times \Pk^1_y}
2\omega_x\w\omega_y = 2.
\]
Since $\hat\rho=\dim Y+\dim Y-3=1\ge 0$, the B\'ezout formula 
\[
\deg_L (Y\bl Y)=\deg_L Y\cdot \deg_L Y
\]
holds,  cf. Proposition~\ref{sugga}. 
Thus
$ Y\bl Y$ must have degree $2\cdot 2=4$. On the other hand, $\cdots$ in \eqref{studsmatta} can only contain a term 
$\mu$ of dimension $\hat\rho=1$, 
since all components of $i_*Y\bullet i_*Y$ must have dimension at least the expected dimension $\hat\rho$,
cf.~\eqref{lbull} and Section~\ref{kongas}.  
Thus $\deg_L\mu=2$. For symmetry reasons it is natural to guess
that 
\begin{equation}\label{bulla}
 Y\bl Y=Y+ \omega_L\w Y.
 \end{equation}
 
 Let us check \eqref{bulla} by means of \eqref{snar1} in Proposition \ref{snarstucken} and \eqref{butter}.
Notice that $d=\dim Y + \dim Y =4$, $\hat\rho= d-3=1$, and $V=Y$ so that
$\dim V=2$. Moreover,  
\begin{multline*}
c(TY)=c(T\Pk^1_x)\wedge c(T\Pk^1_y) = (1+\omega_x)^2 \wedge
(1+\omega_y)^2=(1+2\omega_x)\wedge (1+2\omega_y) =\\
1+2(\omega_x+\omega_y) + 4 \omega_x \wedge \omega_y=
1+2\omega_L + 2\omega_L^2. 
\end{multline*} 
%Here we have used something about Chern forms and
%$c(T\P^N)=(1+\omega)^{N+1}$. 
%First let guess that $Y\bl Y=Y+ \omega_L\w Y$ and apply \eqref{snar1}. 
Assuming \eqref{bulla}, the right hand side of \eqref{snar1} equals 
\begin{multline*}
%\mu_1\di_Y\mu_2=
%\sum_{\ell=1}^{2}
%(1+\omega_L)^{\ell-5}\w (1+2\omega_L + 2\omega_L^2 )\w(Y +
%\omega_L\wedge Y)_\ell
%= \\
%\frac{1}{(1+\omega_L)^4}  (1+2\omega_L + 2\omega_L^2 )
%\omega_L\wedge Y
%+ 
%\frac{1}{(1+\omega_L)^3}  (1+2\omega_L + 2\omega_L^2 )\w Y
%=\\
%\frac{1}{(1+\omega_L)^4}  (1+2\omega_L + 2\omega_L^2 ) (\omega_L + 1+
%\omega_L ) \wedge Y
%=
c(TY) \w \sum_{\ell=1}^{2}
(1+\omega_L)^{\ell-5}\w(Y +
\omega_L\wedge Y)_\ell
= 
c(TY) \w \left (\frac{1}{(1+\omega_L)^4} \w 
\omega_L\wedge Y
+ 
\frac{1}{(1+\omega_L)^3}  \w  Y \right )
=\\
c(TY) \w \frac{1}{(1+\omega_L)^4} (\omega_L + 1+
\omega_L ) \wedge Y
=
\frac{1}{(1+\omega_L)^4} (1+2\omega_L + 2\omega_L^2 )\wedge  (1+
2\omega_L ) \wedge Y 
= \\
\frac{1}{1+4\omega_L + 6\omega_L^2} (1+4\omega_L + 6\omega_L^2 )\wedge
Y = Y = Y\di_Y Y.
\end{multline*}
Hence our guess is correct. 
Clearly one can just as well start with \eqref{butter} and apply
\eqref{snar2}.  By similar computations one then gets \eqref{bulla}, as expected.
%First note that 
%\[
%c(TY)^{-1}=\frac{1}{1+2\omega_L+2\omega_L^2}=
%1+(-2\omega_L -2\omega_L^2) + (-2\omega_L -2\omega_L^2)^2
%= 1-2\omega_L -2\omega_L^2+ 4 \omega_L^2 
%= 1-2\omega_L + 2\omega_L^2.
%\]
 %A similar computation yields 
%\begin{multline*}
%\frac{1}{(1-\omega_L)^3}
%=
%\frac{1}{1-3\omega_L+ 3 \omega_L^2}
%=
%1+ (3\omega_L- 3 \omega_L^2) + (3\omega_L- 3 \omega_L^2)^2
%=
%1 + 3\omega_L- 3 \omega_L^2 + 9 \omega_L^2
%=
%1+ 3\omega_L+ 6\omega_L^2
%\end{multline*}
%\[
%\frac{1}{(1-\omega_L)^3}= 
%1+ 3\omega_L+ 6\omega_L^2, ~~~~~~~~ 
%\frac{1}{(1-\omega_L)^4}= 
%1+ 4\omega_L+ \ldots 
%\]
%Plugging this into \eqref{snar2} we get 
%\begin{multline*}
%\mu_1\bl\mu_2=\sum_{k=0}^{2}
%(1-\omega_L)^{k-5}\big ( (1-2\omega_L + 2\omega_L^2)\wedge Y\big)_k
%=\\
%\frac{1}{(1-\omega_L)^5}\w  2\omega_L^2\w Y 
%+
%\frac{1}{(1-\omega_L)^4} \w (-2\omega_L)\w Y 
%+
%\frac{1}{(1-\omega_L)^3} \w Y 
%=\\
%\big (2\omega_L^2  + (1+4\omega_L)\w (-2\omega_L) + 
%(1+ 3\omega_L+ 6\omega_L^2) \big )\w Y = (1+ \omega_L)\w Y
%\end{multline*} 
%which is the same result we got above. 
%
%
\end{ex}

The following example, which is an elaboration of
\cite[Example~8.10]{aeswy2}, shows that the product $\di_Y$ is not
associative.

\begin{ex}\label{kokong} 
 Consider the hypersurface  $Z=\{x_2x_1^m-x_3^2x_0^{m-1}=0\}$ and the
 hyperplanes $H_2=\{x_2=0\}$ and $H_3=\{x_3=0\}$ in 
 $Y=\Pk^3$. (All products here are taken
in $\Pk^3$; let $\di= \di_{\Pk^3}$ and $\cdot =\cdot_{\Pk^3}$.) 
Since $H_2$ and $Z$ intersect properly, 
\begin{equation}\label{spor1}
%H_2\di_{\P^3} Z=H_2\bullet_{\P^3} Z=H_2\cdot_{\P^3} Z =\\
H_2\di Z=H_2\bullet Z=H_2\cdot Z =
2\{x_2=x_3=0\}+(m-1)\{x_0=x_2=0\}=: 2 A + (m-1)B. 
\end{equation}
Moreover, since $H_3$ and $B$ intersect properly, 
\begin{equation}\label{spor2}
H_3\di B=H_3\cdot B =[b],
\end{equation}
% and the intersection is the
%point $b=[0,1,0,0]$; thus 
%Thus 
%\[
%H_3\di B=H_2\bullet B=H_3\cdot B =[b], 
%\]
where $b=[0,1,0,0]$. 
In \cite[Example~8.10]{aeswy2} is showed that $H_3\bullet A=A$. 
In view of \eqref{potatis} it follows that 
\begin{equation}\label{spor3}
H_3\di A = (1+\omega) \wedge A.
\end{equation}
Indeed $\rho=\dim H_3 + \dim A-n=2+1-3=0$ and $H_3\bullet A=A$ has pure
dimension $1$. 
By \eqref{spor1}, \eqref{spor2} and \eqref{spor3} we conclude 
\begin{equation}\label{spor4}
H_3\di (H_2 \di Z) = 2(1+\omega)\wedge A + (m-1)[b]. 
\end{equation}
Next, since $H_3$ and $H_2$ intersect properly, it follows that 
$H_3\di H_2=H_3\cdot H_2 =A$. 
In \cite[Example~8.10]{aeswy2} we showed that
$A\bullet Z=A + m[a]$, where $a=[1:0:0:0]$. 
As above, we compute $A\di Z$ by applying \eqref{potatis}. To do this, note that
$\rho=\dim A + \dim Z -n=1+2-3=0$ and $V=|A|\cap |Z| =|A|$, so that
$\dim V=1$. Thus
\begin{equation}\label{spor5}
(H_3\di H_2) \di Z = \sum_{\ell=0}^1 (1+\omega)^\ell \wedge (A\bullet Z)_\ell
=m[a]+(1+\omega)\wedge A.
\end{equation}
Clearly the right hand sides in \eqref{spor4} and \eqref{spor5} are different in $\B(\Pk^3)$. 
\end{ex}

\end{document}